\documentclass[12pt]{article}
\usepackage{amsmath,amsfonts,amssymb,amsthm}

\theoremstyle{plain}
\newtheorem{theorem}{Theorem}[section]
\newtheorem{corollary}[theorem]{Corollary}
\newtheorem{lemma}[theorem]{Lemma}
\newtheorem{proposition}[theorem]{Proposition}

\theoremstyle{definition}
\newtheorem{definition}[theorem]{Definition}

\newtheorem{example}[theorem]{Example}
\newtheorem{remark}[theorem]{Remark}

\newtheorem*{openproblem}{Open Problem}

\numberwithin{equation}{section}

\newcommand{\R}{{\mathbb R}}
\newcommand{\N}{{\mathbb N}}

\providecommand{\vint}[1]{\mathchoice
          {\mathop{\vrule width 5pt height 3 pt depth -2.5pt
                  \kern -9pt \kern 1pt\intop}\nolimits_{\kern -5pt{#1}}}
          {\mathop{\vrule width 5pt height 3 pt depth -2.6pt
                  \kern -6pt \intop}\nolimits_{\kern -3pt{#1}}}
          {\mathop{\vrule width 5pt height 3 pt depth -2.6pt
                  \kern -6pt \intop}\nolimits_{\kern -3pt{#1}}}
          {\mathop{\vrule width 5pt height 3 pt depth -2.6pt
                  \kern -6pt \intop}\nolimits_{\kern -3pt{#1}}}}

\newcommand{\eps}{\varepsilon}
\newcommand{\loc}{{\mbox{\scriptsize{loc}}}}

\newcommand{\BV}{\mathrm{BV}}
\newcommand{\liploc}{\mathrm{Lip}_{\mathrm{loc}}}

\DeclareMathOperator{\capa}{Cap}
\DeclareMathOperator{\rcapa}{cap}

\DeclareMathOperator{\dist}{dist}

\DeclareMathOperator{\Lip}{Lip}

\DeclareMathOperator{\supp}{supp}

\begin{document}
\title{A notion of fine continuity for $\BV$ functions \\
on metric spaces
\footnote{
\hfill \break {\it Keywords\,}: bounded variation, metric measure space, quasicontinuity, fine topology, fine continuity, upper hemicontinuity.
}}
\author{Panu Lahti\\
Mathematical Institute, University of Oxford,\\ Andrew Wiles Building,\\
Radcliffe Observatory Quarter, Woodstock Road,\\
Oxford, OX2 6GG.\\
E-mail: {\tt lahti@maths.ox.ac.uk}
}
\maketitle

\begin{abstract} 
In the setting of a metric space equipped with a doubling measure supporting a Poincar\'e inequality, we show that $\BV$ functions are, in the sense of multiple limits, continuous with respect to a $1$-\emph{fine topology}, at almost every point with respect to the codimension $1$ Hausdorff measure.
\end{abstract}

\noindent {\bf Acknowledgment:} The research was funded by a grant from the Finnish Cultural Foundation. The author wishes to thank Nageswari Shan\-muga\-lingam for useful feedback on the paper. He also wishes to thank the
University of Cincinnati, in which part of the research for this paper was conducted, for its kind hospitality.

\newpage
\section{Introduction}

It is known, in the generality of a metric measure space $(X,d,\mu)$ equipped with a doubling measure $\mu$ supporting a Poincar\'e inequality, that Newton-Sobolev functions $u\in N^{1,p}(X)$ are $p$-quasicontinuous. This means that there exists an open set $G\subset X$ of small $p$-capacity such that the restriction $u|_{X\setminus G}$ is continuous, see e.g. \cite{BBS1} or \cite{BB}. From this, one can derive another result, which states that Newton-Sobolev functions are \emph{$p$-finely continuous} at
$p$-quasi every point, that is, almost every point with respect to the $p$-capacity, see \cite{JaBj} or \cite{Kor} or \cite[Theorem 11.40]{BB}. The concept of $p$-fine continuity means continuity with respect to a suitable topology, the $p$-fine topology, which is somewhat stronger than the metric topology. For previous results on fine topology and fine continuity in the Euclidean setting, see also e.g. \cite{Car,HKM,MaZie}.

In \cite{LaSh} it was shown that $\BV$ functions on metric spaces are $1$-quasi\-con\-ti\-nu\-ous in the sense of multiple limits. In this paper we introduce a notion of $1$-fine topology, and show that $\BV$ functions are $1$-finely continuous (that is, continuous with respect to the $1$-fine topology) at $1$-quasi every point, again in the sense of multiple limits. This is given in Theorem~\ref{thm:fine continuity}. Instead of $1$-quasi every point, one may equivalently speak about $\mathcal H$-almost every point, where $\mathcal H$ is the codimension $1$ Hausdorff measure.

Our definition of the $1$-fine topology is based on a concept of $1$-thinness, which is analogous to a concept of $p$-fatness, with $p>1$, given in the metric setting in \cite{BMS} and originally defined in \cite{Lew}. Let us also note that the proofs for fine continuity given in \cite{JaBj} and \cite{Kor} involve the theory of $p$-harmonic functions, for $p>1$. While some results on $1$-harmonic functions, known as functions of least gradient, have been derived in \cite{HKL, HKLS, KKLS}, we do not use this theory, relying on a geometric tool known as the \emph{boxing inequality} instead.

\section{Preliminaries}

In this section we introduce the necessary definitions and assumptions.

In this paper, $(X,d,\mu)$ is a complete metric space equipped
with a Borel regular outer measure $\mu$ satisfying a doubling property, that is,
there is a constant $C_d\ge 1$ such that
\[
0<\mu(B(x,2r))\leq C_d\,\mu(B(x,r))<\infty
\]
for every ball $B=B(x,r)$ with center $x\in X$ and radius $r>0$. We also assume that $X$ consists of at least two points.

In general, $C\ge 1$ will denote a constant whose particular value is not important for the purposes of this
paper, and might differ between
each occurrence. When we want to specify that a constant $C$
depends on the parameters $a,b, \ldots,$ we write $C=C(a,b,\ldots)$. Unless otherwise specified, all constants only 
depend on the doubling constant $C_d$ and the constants $C_P,\lambda$ associated
with the Poincar\'e inequality defined below.

A complete metric space with a doubling measure is proper,
that is, closed and bounded subsets are compact. Since $X$ is proper, for any open set $\Omega\subset X$
we define $\liploc(\Omega)$ to be the space of
functions that are Lipschitz in every $\Omega'\Subset\Omega$.
Here $\Omega'\Subset\Omega$ means that $\Omega'$ is open and that $\overline{\Omega'}$ is a
compact subset of $\Omega$. Other local spaces of functions are defined similarly.

For any set $A\subset X$ and $0<R<\infty$, the restricted spherical Hausdorff content
of codimension $1$ is defined by
\begin{equation}\label{eq:definition of Hausdorff content}
\mathcal{H}_{R}(A):=\inf\left\{ \sum_{i=1}^{\infty}
  \frac{\mu(B(x_{i},r_{i}))}{r_{i}}:\,A\subset\bigcup_{i=1}^{\infty}B(x_{i},r_{i}),\,r_{i}\le R\right\}.
\end{equation}
The codimension $1$ Hausdorff measure of a set $A\subset X$ is given by
\begin{equation*}
  \mathcal{H}(A):=\lim_{R\rightarrow 0^+}\mathcal{H}_{R}(A).
\end{equation*}

The measure theoretic boundary $\partial^{*}E$ of a set $E\subset X$ is the set of points $x\in X$
at which both $E$ and its complement have positive upper density, i.e.
\[
\limsup_{r\to 0^+}\frac{\mu(B(x,r)\cap E)}{\mu(B(x,r))}>0\quad\;
  \textrm{and}\quad\;\limsup_{r\to 0^+}\frac{\mu(B(x,r)\setminus E)}{\mu(B(x,r))}>0.
\]
The measure theoretic interior and exterior of $E$ are defined respectively by
\begin{equation}\label{eq:definition of measure theoretic interior}
I_E:=\left\{x\in X:\,\lim_{r\to 0^+}\frac{\mu(B(x,r)\setminus E)}{\mu(B(x,r))}=0\right\}
\end{equation}
and
\begin{equation}\label{eq:definition of measure theoretic exterior}
O_E:=\left\{x\in X:\,\lim_{r\to 0^+}\frac{\mu(B(x,r)\cap E)}{\mu(B(x,r))}=0\right\}.
\end{equation}
A curve is a rectifiable continuous mapping from a compact interval
into $X$.
A nonnegative Borel function $g$ on $X$ is an upper gradient 
of an extended real-valued function $u$
on $X$ if for all curves $\gamma$ on $X$, we have
\[
|u(x)-u(y)|\le \int_\gamma g\,ds,
\]
where $x$ and $y$ are the end points of $\gamma$. We interpret $|u(x)-u(y)|=\infty$ whenever  
at least one of $|u(x)|$, $|u(y)|$ is infinite. Upper gradients were originally introduced in~\cite{HK}.

For $1\le p<\infty$, we consider the following norm
\[
\Vert u\Vert_{N^{1,p}(X)}:=\Vert u\Vert_{L^p(X)}+\inf \Vert g\Vert_{L^p(X)},
\]
with the infimum taken over all upper gradients $g$ of $u$. 
The substitute for the Sobolev space $W^{1,p}(\R^n)$ in the metric setting is the Newton-Sobolev space
\[
N^{1,p}(X):=\{u:\|u\|_{N^{1,p}(X)}<\infty\}.
\]
For more on Newton-Sobolev spaces, we refer to~\cite{S, BB, HKST}.

Next we recall the definition and basic properties of functions
of bounded variation on metric spaces, see \cite{M}. See also e.g. \cite{AFP, EvaG92, Giu84, Zie89} for the classical 
theory in the Euclidean setting.
For $u\in L^1_{\loc}(X)$, we define the \emph{total variation} of $u$ on $X$ to be 
\[
\|Du\|(X):=\inf\left\{\liminf_{i\to\infty}\int_X g_{u_i}\,d\mu:\, u_i\in \Lip_{\loc}(X),\, u_i\to u\textrm{ in } L^1_{\loc}(X)\right\},
\]
where each $g_{u_i}$ is an upper gradient of $u_i$.
The total variation is clearly lower semicontinuous with respect to convergence in $L^1_{\loc}(X)$.
We say that a function $u\in L^1(X)$ is \emph{of bounded variation}, 
and denote $u\in\BV(X)$, if $\|Du\|(X)<\infty$.
By replacing $X$ with an open set $\Omega\subset X$ in the definition of the total variation, we can define $\|Du\|(\Omega)$.
A $\mu$-measurable set $E\subset X$ is said to be of \emph{finite perimeter} if $\|D\chi_E\|(X)<\infty$, where $\chi_E$ is the characteristic function of $E$.
The perimeter of $E$ in $\Omega$ is also denoted by
\[
P(E,\Omega):=\|D\chi_E\|(\Omega).
\]
For any Borel sets $E_1,E_2\subset X$ we have by \cite[Proposition 4.7]{M}
\begin{equation}\label{eq:Caccioppoli sets form an algebra}
P(E_1\cup E_2,X)\le P(E_1,X)+P(E_2,X). 
\end{equation}

We will assume throughout that $X$ supports a $(1,1)$-Poincar\'e inequality,
meaning that there exist constants $C_P>0$ and $\lambda \ge 1$ such that for every
ball $B(x,r)$, every locally integrable function $u$ on $X$,
and every upper gradient $g$ of $u$,
we have 
\[
\vint{B(x,r)}|u-u_{B(x,r)}|\, d\mu 
\le C_P r\vint{B(x,\lambda r)}g\,d\mu,
\]
where 
\[
u_{B(x,r)}:=\vint{B(x,r)}u\,d\mu :=\frac 1{\mu(B(x,r))}\int_{B(x,r)}u\,d\mu.
\]
By applying the Poincar\'e inequality to approximating locally Lipschitz functions in the definition of the total variation, 
we get the following $(1,1)$-Poincar\'e inequality for $\BV$ functions. There exists a constant $C$
such that for every ball $B(x,r)$ and every 
$u\in L^1_{\loc}(X)$, we have
\[
\vint{B(x,r)}|u-u_{B(x,r)}|\,d\mu
\le Cr\, \frac{\Vert Du\Vert (B(x,\lambda r))}{\mu(B(x,\lambda r))}.
\]
For $\mu$-measurable sets $E\subset X$, the above can be written as
\begin{equation}\label{eq:relative isoperimetric inequality}
\min\{\mu(B(x,r)\cap E),\,\mu(B(x,r)\setminus E)\}\le CrP(E,B(x,\lambda r)).
\end{equation}

For $1\le p<\infty$, the $p$-capacity of a set $A\subset X$ is given by
\begin{equation}\label{eq:definition of p-capacity}
 \capa_p(A):=\inf \Vert u\Vert_{N^{1,p}(X)},
\end{equation}
where the infimum is taken over all functions $u\in N^{1,p}(X)$ such that $u\ge 1$ in $A$.
If a property holds for all points outside a set of $p$-capacity zero, we say that it holds for $p$-quasi every point, or $p$-quasi\-everywhere.

The relative $p$-capacity of a set $A\subset X$ with respect to an open set $\Omega\subset X$ is given by
\[
\rcapa_p(A,\Omega):=\inf \int_{\Omega} g_u^p\,d\mu,
\]
where the infimum is taken over functions $u\in N^{1,p}(X)$ and upper gradients $g_u$ of $u$ such that $u\ge 1$ in $A$  and $u=0$ in $X\setminus \Omega$.
For basic properties satisfied by capacities, such as monotonicity and countable subadditivity, see e.g. \cite{BB}.
The $\BV$-capacity of a set $A\subset X$ is
\begin{equation}\label{eq:definition of BV capacity}
\capa_{\BV}(A):=\inf\Vert u\Vert_{\BV(X)},
\end{equation}
where the infimum is taken over functions $u\in\BV(X)$ that satisfy $u\ge 1$ in a neighborhood of $A$. Note that we understand $\BV$ functions to be $\mu$-equivalence classes, whereas we understand Newton-Sobolev functions to be defined everywhere (even though $\Vert \cdot\Vert_{N^{1,p}(X)}$ is then only a seminorm).

Given a set $E\subset X$ of finite perimeter, for $\mathcal H$-almost every $x\in \partial^*E$ we have
\begin{equation}\label{eq:definition of gamma}
\gamma \le \liminf_{r\to 0^+} \frac{\mu(E\cap B(x,r))}{\mu(B(x,r))} \le \limsup_{r\to 0^+} \frac{\mu(E\cap B(x,r))}{\mu(B(x,r))}\le 1-\gamma
\end{equation}
where $\gamma \in (0,1/2]$ only depends on the doubling constant and the constants in the Poincar\'e inequality, 
see~\cite[Theorem 5.4]{A1}.
For an open set $\Omega\subset X$ and a set $E\subset X$ of finite perimeter, we know that 
\begin{equation}\label{eq:def of theta}
P(E,\Omega)=\int_{\partial^{*}E\cap \Omega}\theta_E\,d\mathcal H,
\end{equation}
where
$\theta_E\colon X\to [\alpha,C_d]$ with $\alpha=\alpha(C_d,C_P,\lambda)>0$, see \cite[Theorem 5.3]{A1} 
and \cite[Theorem 4.6]{AMP}.

The jump set of $u\in\BV(X)$ is the set
\[
S_{u}:=\{x\in X:\, u^{\wedge}(x)<u^{\vee}(x)\},
\]
where $u^{\wedge}(x)$ and $u^{\vee}(x)$ are the lower and upper approximate limits of $u$ defined respectively by
\[
u^{\wedge}(x):
=\sup\left\{t\in\overline\R:\,\lim_{r\to 0^+}\frac{\mu(B(x,r)\cap\{u<t\})}{\mu(B(x,r))}=0\right\}
\]
and
\[
u^{\vee}(x):
=\inf\left\{t\in\overline\R:\,\lim_{r\to 0^+}\frac{\mu(B(x,r)\cap\{u>t\})}{\mu(B(x,r))}=0\right\}.
\]
In the Euclidean setting, results on the fine properties of $\BV$ functions can be formulated in terms of $u^{\wedge}$ and $u^{\vee}$, but in the metric setting, we need to consider a larger number of jump values. The reason for this is explained in Example \ref{ex:one dimensional space}.
Thus we define the functions $u^l$, $l=1,\ldots,n:=\lfloor 1/\gamma\rfloor$, as 
follows: $u^1:=u^{\wedge}$, $u^n:=u^{\vee}$, 
and for $l=2,\ldots,n-1$ we define inductively
\begin{equation}\label{eq:definition of n jump values}
u^{l}(x):=\sup\left\{t\in\overline{\R}:\,\lim_{r\to 0^+}\frac{\mu(B(x,r)\cap \{u^{l-1}(x)+\delta<u<t\})}{\mu(B(x,r))}=0\ \ \forall\, \delta>0\right\}
\end{equation}
provided $u^{l-1}(x)<u^\vee(x)$, and otherwise we set 
$u^l(x)=u^{\vee}(x)$. 
It can be shown that each $u^l$ is a Borel function,
and $u^{\wedge}=u^1\le \ldots \le u^n = u^{\vee}$.

We have the following notion of quasicontinuity for $\BV$ functions.

\begin{theorem}[{\cite[Theorem~1.1]{LaSh}}]\label{thm:quasicontinuity}
Let $u\in\BV(X)$ and let $\eps>0$. Then there exists an open set $G\subset X$ with $\capa_1(G)<\eps$ such that if $y_k\to x$ 
with $y_k,x\in X\setminus G$, then 
\[
\min_{l_2\in \{1,\ldots,n\}} |u^{l_1}(y_k)-u^{l_2}(x)|\to 0
\]
for each $l_1=1,\ldots,n$.
\end{theorem}

\begin{remark}\label{rmk:capacities and Hausdorff contents}
The $1$-capacity and Hausdorff contents are closely related:
it follows from~\cite[Theorem~4.3, Theorem~5.1]{HaKi} that
$\capa_1(A)=0$ if and only if $\mathcal{H}(A)=0$. 
Moreover, from \cite[Lemma 3.4]{KKST3} it follows that $\capa_1(A)\le 2C_d\mathcal H_1(A)$ for any set $A\subset X$.
On the other hand,
by combining~\cite[Theorem~4.3]{HaKi} and the proof of~\cite[Theorem 5.1]{HaKi}, 
we obtain that
\[
\mathcal H_{\eps}(A)\le C(C_d,C_P,\lambda,\eps)\capa_1(A)
\]
for 
any $A\subset X$ and $\eps>0$. Thus we can also control the size of the ``exceptional set'' $G$ in Theorem~\ref{thm:quasicontinuity} 
and elsewhere by its $\mathcal H_{\eps}$-measure, for arbitrarily small $\eps>0$.
\end{remark}

\section{Rigidity results for the $1$-capacity}

In order to prove our main result, Theorem \ref{thm:fine continuity}, we need to be able to modify the ``exceptional set'' $G$ of Theorem~\ref{thm:quasicontinuity} in a suitable way. In this section we show that sets can be enlarged in two different ways without increasing the $1$-capacity significantly.

It is known that $\capa_1$ is an \emph{outer capacity}, meaning that
\[
\capa_1(G)=\inf\{\capa_1(U):\,U\textrm{ is open and }U\supset G \}
\]
for any $G\subset X$,
see e.g. \cite[Theorem 5.31]{BB}. The following result is in the same spirit as this fact.

\begin{lemma}\label{lem:G has finite perimeter}
For any $G\subset X$, we can find an open set $U\supset G$ with $\capa_1(U)\le C\capa_1(G)$ and $P(U,X)\le C\capa_1(G)$.
\end{lemma}
\begin{proof}
We can assume that $\capa_1(G)<\infty$. According to Remark~\ref{rmk:capacities and Hausdorff contents}, we have $\mathcal H_{1/2}(G)\le C\capa_1(G)$. Take a covering $\{B(x_i,r_i)\}_{i\in\N}$ of the set $G$ with $r_i\le 1/2$ and
\[
\sum_{i\in\N}\frac{\mu(B(x_i,r_i))}{r_i}\le C\capa_1(G).
\]
By \cite[Lemma 6.2]{KKST}, for each $i\in\N$ there exists a radius $\widetilde{r}_i\in [r_i,2r_i]$ such that
\[
P(B(x_i,\widetilde{r}_i),X)\le C_d\frac{\mu(B(x_i,\widetilde{r}_i))}{\widetilde{r}_i}.
\]
By using the lower semicontinuity and subadditivity of perimeter, recall \eqref{eq:Caccioppoli sets form an algebra}, we get
\begin{align*}
P\left(\bigcup_{i\in\N}B(x_i,\widetilde{r}_i),X\right)
&\le \sum_{i\in\N}P(B(x_i,\widetilde{r}_i),X)\\
&\le C_d\sum_{i\in\N}\frac{\mu(B(x_i,\widetilde{r}_i))}{\widetilde{r}_i}
\le C\capa_1(G).
\end{align*}
So we can define $U:=\bigcup_{i\in\N}B(x_i,\widetilde{r}_i)$, with $U\supset G$, and then $P(U,X)\le C\capa_1(G)$ and $\mathcal H_{1}(U)\le C\capa_1(G)$, so that also $\capa_1(U)\le C\capa_1(G)$ by Remark \ref{rmk:capacities and Hausdorff contents}.
\end{proof}

In proving our second rigidity result, we will use \emph{discrete convolutions} of $\BV$ functions.
By the doubling property of the measure $\mu$, given any scale $R>0$ we can pick a covering of the space $X$ by balls $B(x_j,R)$, such that suitable dilated balls, say $B(x_j,10\lambda R)$, have bounded overlap. More precisely, each $B(x_k,10\lambda R)$ meets at most $C$ balls $B(x_j,10\lambda R)$.
Given such a covering, 
we can take a partition of unity $\{\phi_j\}_{j=1}^{\infty}$ subordinate to it, such that $0\le \phi_j\le 1$, each 
$\phi_j$ is a $C/R$-Lipschitz function, and $\supp(\phi_j)\subset B(x_j,2 R)$ for each 
$j\in\N$ (see e.g. \cite[Theorem 3.4]{BBS07}). Finally, we can define a discrete convolution $v$ of 
any $u\in \BV(X)$ with respect to the covering by
\[
v:=\sum_{j=1}^{\infty}u_{B(x_j,5R)}\phi_j.
\]
We know that $v$ has an upper gradient
\[
g=C\sum_{j=1}^{\infty}\chi_{B_j}\frac{\Vert Du\Vert(B(x_j,10\lambda R))}{\mu(B(x_j,R))},
\]
see e.g. the proof of \cite[Proposition 4.1]{KKST2}, and so by the bounded overlap of the balls $B(x_j,10\lambda R)$, we have $\Vert g\Vert_{L^1(X)}\le C\Vert Du\Vert(X)$. We also have $\Vert v\Vert_{\BV(X)}\le \Vert v\Vert_{N^{1,1}(X)}$ (since Lipschitz functions are dense in the class $N^{1,1}(X)$, see e.g. \cite[Theorem 5.1]{BB}), and thus
\begin{equation}\label{eq:estimate for BV norm of discrete convolution}
\Vert v\Vert_{\BV(X)}\le C\Vert u\Vert_{\BV(X)}.
\end{equation}
If $u\in\BV(X)$ and each $v_i$, $i\in\N$, is a discrete convolution of $u$ at scale $1/i$, we know that for some $\widetilde{\gamma}=\widetilde{\gamma}(C_d,C_P,\lambda)\in (0,1/2]$,
\begin{equation}\label{eq:pointwise convergence}
\begin{split}
&(1-\widetilde{\gamma}) u^{\wedge}(y)+\widetilde{\gamma} u^{\vee}(y) \le \liminf_{i\to\infty}v_i(y)\\
&\qquad \le \limsup_{i\to\infty}v_i(y)\le \widetilde{\gamma} u^{\wedge}(y)+(1-\widetilde{\gamma}) u^{\vee}(y)
\end{split}
\end{equation}
for $\mathcal H$-almost every $y\in X$, see \cite[Proposition 4.1]{KKST2}.

Recall the definitions of the $1$-capacity and the $\BV$-capacity from \eqref{eq:definition of p-capacity} and \eqref{eq:definition of BV capacity}. By \cite[Theorem 4.3]{HaKi} we know that
\begin{equation}\label{eq:comparability of capacities}
\capa_{\BV}(A)\le \capa_1(A)\le C\capa_{\BV}(A)
\end{equation}
for any $A\subset X$.

Now we prove the following rigidity result for the $1$-capacity. Recall from \eqref{eq:definition of measure theoretic interior} the definition of the measure theoretic interior $I_G$ of a set $G$.

\begin{proposition}\label{prop:capacity and adding points to set}
Let $G\subset X$ be an arbitrary set.
Then
\[
\capa_1(G\cup I_G\cup\partial^*G)\le C\capa_1(G).
\]
\end{proposition}
\begin{proof}
By \eqref{eq:comparability of capacities} it is enough to prove this for $\capa_{\BV}$ instead of $\capa_1$. We can assume that $\capa_{\BV}(G)<\infty$. Fix $\eps>0$ and choose $u\in \BV(X)$ with $u\ge 0$, $u\ge 1$ in a neighborhood of $G$, and $\Vert u\Vert_{\BV(X)}\le \capa_{\BV}(G)+\eps$. Let each $v_i\in \liploc(X)$, $i\in\N$, be a discrete convolution of $u$ at scale $1/i$, and let $N\subset X$ be the set where \eqref{eq:pointwise convergence} fails, so that $\mathcal H(N)=0$. Thus we have (recall Remark \ref{rmk:capacities and Hausdorff contents})
\[
\capa_{\BV}(G\cup I_G\cup\partial^*G\setminus N)=\capa_{\BV}(G\cup I_G\cup\partial^*G).
\]
Clearly $u^{\wedge}\ge 1$ in $G\cup I_G$, and $u^{\vee}\ge 1$ in $\partial^*G$, so that
\[
(1-\widetilde{\gamma}) u^{\wedge}(y)+\widetilde{\gamma} u^{\vee}(y)\ge \widetilde{\gamma}
\]
for every $y\in G\cup I_G\cup \partial^*G$, and so by \eqref{eq:pointwise convergence},
\[
\liminf_{i\to\infty}v_i(y)\ge \widetilde{\gamma}\qquad\textrm{for every}\ \ y\in G\cup I_G\cup \partial^*G\setminus N.
\]
Define the sets
\[
G_i:=\{x\in G\cup I_G\cup \partial^*G\setminus N:\, v_j(x)>\widetilde{\gamma}/2\ \ \textrm{for all }j\ge i\},\quad i\in\N.
\]
Now we have $G_1\subset G_2\subset \ldots$ and $\bigcup_{i\in\N}G_i=G\cup I_G\cup \partial^*G\setminus N$. Since discrete convolutions are continuous, clearly $v_i>\widetilde{\gamma}/2$ in a neighborhood of $G_i$ and so we can use $2v_i/\widetilde{\gamma}$ to estimate the $\BV$-capacity of $G_i$. Furthermore, by \cite[Theorem~3.4]{HaKi} we know that the $\BV$-capacity is continuous with respect to increasing sequences of sets, and so we get
\begin{align*}
\capa_{\BV}(G\cup I_G\cup \partial^*G)
&=\capa_{\BV}(G\cup I_G\cup \partial^*G\setminus N)
=\capa_{\BV}\left(\bigcup_{i\in\N}G_i\right)\\
&=\lim_{i\to\infty}\capa_{\BV}(G_i)
\le \frac{2}{\widetilde{\gamma}}\liminf_{i\to\infty}\Vert v_i\Vert_{\BV(X)}\\
&\overset{\eqref{eq:estimate for BV norm of discrete convolution}}{\le} C\Vert u\Vert_{\BV(X)}
\le C(\capa_{\BV}(G)+\eps)
\end{align*}
by the choice of $u$.
By letting $\eps\to 0$, we get the result.
\end{proof}

\section{The $1$-fine topology}

Our result on $1$-fine continuity will be based on a concept of a fine topology on the space.
Let us first consider some background concerning the case $1<p<\infty$. The following definitions and facts are given in \cite{JaBj} and \cite[Section 11.6]{BB}. A set $A\subset X$ is $p$-thin at $x\in X$ if
\[
\int_0^1\left(\frac{\rcapa_p(A\cap B(x,r),B(x,2r))}{\rcapa_p(B(x,r),B(x,2r))}\right)^{1/(p-1)}\,\frac{dr}{r}<\infty.
\]
A set $U\subset X$ is $p$-finely open if $X\setminus U$ is $p$-thin at every $x\in U$.
The collection of $p$-finely open sets is a topology on $X$, called the $p$-fine topology.
Let $\overline{G}^p$ be the $p$-fine closure of $G\subset X$ (smallest $p$-finely closed set containing $G$).
For an open set $\Omega\subset X$ with $\capa_p(X\setminus\Omega)>0$ and $G\Subset \Omega$, we have
\[
\rcapa_p(\overline{G}^p,\Omega)=\rcapa_p(G,\Omega).
\]
A $p$-finely closed set is measure theoretically closed, as follows from \cite[Corollary 11.25]{BB}, and thus the measure theoretic closure $G\cup I_G\cup \partial^*G$ is a subset of $\overline{G}^p$. Thus in the case $p>1$, a stronger result than Proposition \ref{prop:capacity and adding points to set} holds.

In a similar vein, according to \cite[Definition 1.1]{BMS} (which is based on \cite{Lew}) a set $A\subset X$ is said to be $p$-fat at a point $x\in X$ if
\[
\limsup_{r\to 0^+}\frac{\rcapa_p(A\cap B(x,r),B(x,2r))}{\rcapa_p(B(x,r),B(x,2r))}>0.
\]
By \cite[Proposition 6.16]{BB} we know that for small $r>0$ and $1\le p<\infty$, $\rcapa_p(B(x,r),B(x,2r))$ is comparable to $\mu(B(x,r))/r^p$. This motivates the following definition.

\begin{definition}
We say that $A\subset X$ is $1$-thin at the point $x\in X$ if
\[
\lim_{r\to 0^+}r\frac{\rcapa_1(A\cap B(x,r),B(x,2r))}{\mu(B(x,r))}=0.
\]
We also say that a set $U\subset X$ is $1$-finely open if $X\setminus U$ is $1$-thin at every $x\in U$.
\end{definition}


\begin{lemma}
The collection of $1$-finely open sets is a topology on $X$ (called the $1$-fine topology).
\end{lemma}

\begin{proof}
Let $\{U_i\}_{i\in I}$ be any collection of $1$-finely open sets, and let $x\in \bigcup_{i\in I}U_i$. Then $x\in U_j$ for some $j\in I$. Thus
\begin{align*}
&\limsup_{r\to 0^+}r\frac{\rcapa_1\left(B(x,r)\setminus \bigcup_{i\in I}U_i,B(x,2r)\right)}{\mu(B(x,r))}\\
&\qquad\quad\le 
\limsup_{r\to 0^+}r\frac{\rcapa_1(B(x,r)\setminus U_j,B(x,2r))}{\mu(B(x,r))}=0
\end{align*}
by the fact that $U_j$ is $1$-finely open. Thus $\bigcup_{i\in I}U_i$ is a $1$-finely open set.
Next let $U_1,\ldots,U_k$ be $1$-finely open sets, with $k\in\N$, and suppose $x\in\bigcap_{i=1}^k U_i$. Then by the subadditivity of capacity
\begin{align*}
&\limsup_{r\to 0^+}r\frac{\rcapa_1 \left(B(x,r)\setminus \bigcap_{i=1}^k U_i,B(x,2r)\right)}{\mu(B(x,r))}\\
&\qquad\quad\le \limsup_{r\to 0^+}\sum_{i=1}^k r\frac{\rcapa_1(B(x,r)\setminus U_i,B(x,2r))}{\mu(B(x,r))}\\
&\qquad\qquad\quad\quad \le\sum_{i=1}^k \limsup_{r\to 0^+}r\frac{\rcapa_1(B(x,r)\setminus U_i,B(x,2r))}{\mu(B(x,r))}=0
\end{align*}
by the fact that each $U_i$ is $1$-finely open. Thus $\bigcap_{i=1}^k U_i$ is a $1$-finely open set.
\end{proof}

Let $\overline{G}^1$ be the $1$-fine closure of $G\subset X$ (smallest $1$-finely closed set containing $G$). In the case $p>1$, a crucial step in showing that Newton-Sobolev functions are $p$-finely continuous (i.e. continuous with respect to the $p$-fine topology) at $p$-quasi every point is showing that $\rcapa_p(\overline{G}^p,\Omega)=\rcapa_p(G,\Omega)$, see the discussion earlier in this section.

\begin{openproblem}
Is it true that
$
\capa_1 (\overline{G}^1)=\capa_1 (G)
$
for every $G\subset X$?
\end{openproblem}


For us it will be enough to have a weaker result that we prove in Proposition \ref{prop:capacity of fine closure}.
Following \cite{KKST3}, we first prove the following local version of the \emph{boxing inequality}.

\begin{lemma}\label{lem:boxing inequality}
Let $x\in X$, let $r>0$, and let $G\subset X$ be a $\mu$-measurable set with
\begin{equation}\label{eq:little G}
\frac{\mu(G\cap B(x,2r))}{\mu(B(x,2r))}\le \frac{1}{2C_d^{\lceil\log_2 ( 128\lambda)\rceil}}.
\end{equation}
Then $\rcapa_1 (I_G\cap B(x,r),B(x,2r))\le CP(G,B(x,2 r))$.
\end{lemma}

\begin{proof}
Fix $y\in I_G\cap B(x,r)$. Since $y\in I_G$, there exists $s\in (0,r/32\lambda)$ such that
\[
\frac{\mu(G\cap B(y,s))}{\mu(B(y,s))}>\frac 12.
\]
On the other hand, for all $t\in (r/32\lambda,r/16\lambda)$ we have $B(x,2r)\subset B(y,128\lambda t)$ and then
\[
\frac{\mu(G\cap B(y,t))}{\mu(B(y,t))}\le C_d^{\lceil\log_2 ( 128\lambda)\rceil}\frac{\mu(G\cap B(x,2r))}{\mu(B(x,2r))}\le \frac 12
\]
by \eqref{eq:little G}. Thus by repeatedly doubling the radius $s$, we eventually obtain a radius $t_y\in (0,r/16\lambda)$ such that
\[
\frac{1}{2C_d}<\frac{\mu(G\cap B(y,t_y))}{\mu(B(y,t_y))}\le\frac{1}{2}.
\]
By the relative isoperimetric inequality \eqref{eq:relative isoperimetric inequality}, this implies that
\begin{equation}\label{eq:rel isop ineq in boxing proof}
\mu(B(y,t_y))\le C\mu(G\cap B(y,t_y))\le Ct_y P(G,B(y,\lambda t_y)).
\end{equation}
Define the function
\begin{equation}\label{eq:cutoff functions for capacity}
w(z):=\max\left\{0,1-\frac{\dist(z,B(y,5\lambda t_y))}{5\lambda t_y}\right\},
\end{equation}
so that $w=1$ in $B(y,5\lambda t_y)$ and $w=0$ outside $B(y,10\lambda t_y)$.
Note that $w$ has an upper gradient $g:=\frac{1}{5\lambda t_y}\chi_{B(y,10\lambda t_y)\setminus B(y,5\lambda t_y)}$. Then since $B(y,10\lambda t_y)\subset B(x,2r)$,
\[
\rcapa_1(B(y,5\lambda t_y),B(x,2r))
\le \int_{B(y,10\lambda t_y)} g\,d\mu\le\frac{\mu(B(y,10\lambda t_y))}{5\lambda t_y}.
\]
Take a covering $\{B(y,\lambda t_y)\}_{y\in I_G\cap B(x,r)}$. By the 5-covering theorem, we can choose a countable disjoint collection $\{B(y_i, \lambda t_i)\}_{i\in\N}$ such that the balls $B(y_i,5 \lambda t_i)$ cover $I_G\cap B(x,r)$. Then we have by the countable subadditivity of capacity
\begin{align*}
\rcapa_1 (I_G\cap B(x,r),B(x,2r))
&\le \sum_{i\in\N}\rcapa_1(B(y_i,5 \lambda t_i),B(x,2r))\\
&\le \sum_{i\in\N}\frac{\mu(B(y_i,10\lambda t_i))}{5\lambda t_i}\\
&\le C\sum_{i\in\N}\frac{\mu(B(y_i,t_i))}{t_i}\\
&\overset{\eqref{eq:rel isop ineq in boxing proof}}{\le} C\sum_{i\in\N} P(G,B(y_i,\lambda t_i))\\
&\le P(G,B(x,2 r)).
\end{align*}
\end{proof}

It is easy to see that for any set $A\subset X$ and any ball $B(x,r)$,
\begin{equation}\label{eq:capacity and Hausdorff measure}
\rcapa_1(A\cap B(x,r),B(x,2r))\le C \mathcal H(A\cap B(x,r)).
\end{equation}
This can be deduced by using suitable cutoff functions similar to those given in \eqref{eq:cutoff functions for capacity}.

\begin{proposition}\label{prop:capacity of fine closure}
For any $G\subset X$ we have $\capa_1(\overline{G}^1)\le C\capa_1(G)$. 
\end{proposition}

\begin{proof}
We can assume that $\capa_1(G)<\infty$.
First assume also that $G$ is open and that $P(G,X)<\infty$.
By \cite[Theorem 2.4.3]{AT} we know that if $\nu$ is a Radon measure on $X$, $t>0$, and $A\subset X$ is a Borel set for which we have
\[
\limsup_{r\to 0^+}r\frac{\nu(B(x,r))}{\mu(B(x,r))}\ge t
\]
for all $x\in A$, then $\nu(A)\ge t\mathcal H(A)$.
Since $G$ is of finite perimeter, we have $\mathcal H(\partial^*G)<\infty$ by \eqref{eq:def of theta}. By using \eqref{eq:capacity and Hausdorff measure} and the above density result with $\nu=\mathcal H|_{\partial^*G}$, we get
\begin{equation}\label{eq:estimate for density of measure theoretic boundary of G}
\limsup_{r\to 0^+}r\frac{\rcapa_1(\partial^*G\cap B(x,r),B(x,2r))}{\mu(B(x,2r))}
\le C\limsup_{r\to 0^+}r\frac{\mathcal H(\partial^*G\cap B(x,2r))}{\mu(B(x,2r))}=0
\end{equation}
for $\mathcal H$-almost every $x\in X\setminus \partial^* G$, that is,
for every $x\in X\setminus (\partial^*G\cup N)$ with $\mathcal H(N)=0$.

By Lemma \ref{lem:boxing inequality}, if $x\in X$ and $r>0$ satisfy
\[
\frac{\mu(G\cap B(x,2r))}{\mu(B(x,2r))}\le \frac{1}{2C_d^{\lceil\log_2 ( 128\lambda)\rceil}},
\]
then $\rcapa_1 (I_G\cap B(x,r),B(x,2r))\le CP(G,B(x,2 r))$.
Thus we get for all $x\in X\setminus (I_G\cup \partial^*G\cup N)$
\begin{align*}
\limsup_{r\to 0^+}r\frac{\rcapa_1 (I_G\cap B(x,r),B(x,2r))}{\mu(B(x,r))}
&\le C\limsup_{r\to 0^+}r\frac{P(G,B(x,2 r))}{\mu(B(x,r))}\\
&\overset{\eqref{eq:def of theta}}{\le} C\limsup_{r\to 0^+}r\frac{\mathcal H(\partial^*G\cap B(x,2 r))}{\mu(B(x,r))}\\
&\overset{\eqref{eq:estimate for density of measure theoretic boundary of G}}{=} 0.
\end{align*}
By combining this with \eqref{eq:estimate for density of measure theoretic boundary of G}, we have
\[
\limsup_{r\to 0^+}r\frac{\rcapa_1 ((I_G\cup \partial^*G)\cap B(x,r), B(x,2r))}{\mu(B(x,r))}
=0
\]
for all $x\in X\setminus (I_G\cup \partial^*G\cup N)$.
Since $G$ is open, $G\subset I_G$.
Thus $I_G\cup \partial^*G\cup N\supset G$ is a $1$-finely closed set, so that $\overline{G}^1\subset  I_G\cup \partial^*G\cup N$.
By Proposition \ref{prop:capacity and adding points to set}, we have
\[
\capa_1(I_G\cup \partial^* G\cup N)=\capa_1(I_G\cup \partial^* G)\le C\capa_1(G).
\]
Thus we have the result when $G$ is open and of finite perimeter.
In the general case, by Lemma \ref{lem:G has finite perimeter} we can choose an open set $U\supset G$ with $\capa_1(U)\le C\capa_1(G)$ and $P(U,X)<\infty$. Thus we have
\[
\capa_1(\overline{G}^1)\le \capa_1(\overline{U}^1)\le C\capa_1(U)\le C\capa_1(G).
\]
\end{proof}

\section{Fine continuity}

Since $\BV$ functions can have multiple jump values $u^1,\ldots,u^n$ in their jump sets (recall the definition from \eqref{eq:definition of n jump values}), we need to consider a notion of continuity for set-valued functions.

\begin{definition}
Let $\mathcal U$ be a topology on $X$. We say that the function $y\mapsto \{u^1(y),\ldots,u^n(y)\}$ is \emph{upper hemicontinuous} with respect to $\mathcal U$ at the point $x$ if for every $\eps>0$, there exists $U\in \mathcal U$ with $x\in U$ such that
\[
\min_{l_2\in\{1,\ldots,n\}} |u^{l_1}(y)-u^{l_2}(x)|<\eps
\]
for each $l_1= 1,\ldots,n$ and all $y\in U$.
\end{definition}

Now we can prove the main result of this paper.

\begin{theorem}\label{thm:fine continuity}
Let $u\in\BV(X)$. Then the function $y\mapsto \{u^1(y),\ldots,u^n(y)\}$ is $1$-finely upper hemicontinuous, i.e. upper hemicontinuous with respect to the $1$-fine topology, at $\mathcal H$-almost every $x\in X$.
\end{theorem}

\begin{proof}
Take sets $G_i\subset X$ with $\capa_1(G_i)< 1/i$, $i\in\N$, as given by our quasicontinuity-type result, Theorem \ref{thm:quasicontinuity}. Then also
\[
\capa_1(\overline{G_i}^1)< C/i
\]
by Proposition \ref{prop:capacity of fine closure}. For $1$-quasi every and thus for $\mathcal H$-almost every $x\in X$, we have $x\notin \bigcap_{i\in\N}\overline{G_i}^1$. Fix such $x$, so that $x\notin \overline{G_j}^1$ for some $j\in\N$, and fix $\eps>0$. Theorem \ref{thm:quasicontinuity} gives a radius $r>0$ such that 
\[
\min_{l_2\in\{1,\ldots,n\}} |u^{l_1}(y)-u^{l_2}(x)|<\eps
\]
for each $l_1=1,\ldots,n$ and all $y\in B(x,r)\setminus G_j$, in particular for all $y\in B(x,r)\setminus \overline{G_j}^1$. But $B(x,r)\setminus \overline{G_j}^1$ is a $1$-finely open set containing $x$. Thus we have the result.
\end{proof}

\begin{corollary}
Let $u\in N^{1,1}(X)$. Then $u$ is $1$-finely continuous at $1$-quasi every $x\in X$.
\end{corollary}

Recall that we understand functions in the class $N^{1,1}(X)$ to be defined everywhere, unlike $\BV$ functions that are defined only up to sets of $\mu$-measure zero.

\begin{proof}
Since Lipschitz functions are dense in $N^{1,1}(X)$, see \cite{BBS1} or \cite[Theorem 5.1]{BB}, we have $N^{1,1}(X)\subset \BV(X)$, so that Theorem \ref{thm:fine continuity} applies to $u$. By \cite[Theorem 4.1, Remark 4.2]{KKST3}, there exists $N\subset X$ with $\capa_1(N)=\mathcal H(N)=0$ such that every $x\in X\setminus N$ is a Lebesgue point, that is,
\[
\lim_{r\to 0^+}\vint{B(x,r)}|u-u(x)|\,d\mu=0.
\]
Thus $u(x)=u^1(x)=\ldots =u^n(x)$ for every such $x$. Assume that the function $y\mapsto \{u^1(y),\ldots,u^n(y)\}$ is $1$-finely upper hemicontinuous at $x\in X\setminus N$, which is true for $\mathcal H$-almost every and thus $1$-quasi every point $x\in X$. Let $\eps>0$. By Theorem \ref{thm:fine continuity} there exists a $1$-finely open set $U\ni x$ such that
\[
\min_{l_2\in\{1,\ldots,n\}} |u^{l_1}(y)-u^{l_2}(x)|<\eps
\]
for each $l_1= 1,\ldots,n$ and all $y\in U$. Then $U\setminus N$ is a $1$-finely open set containing $x$, and $|u(y)-u(x)|<\eps$ for all $y\in U$.
\end{proof}

Now consider the following.
We know (see e.g. \cite[Remark 5.9.2]{Zie89}) that if $u\in L^1(X)$ has a Lebesgue point at $x\in X$, i.e.
\[
\lim_{r\to 0^+}\,\vint{B(x,r)}|u-u^{\vee}(x)|\,d\mu=0,
\]
then there exists a set $A_x\ni x$ with density $1$ at $x$, such that $u^{\vee}|_{A_x}$ is continuous (instead of $u^{\vee}$ we could consider some other pointwise representative).
Similarly, by using the analogs of Lebesgue's differentiation theorem for $\BV$ functions, see \cite[Theorem 5.3]{LaSh}, we obtain the following.
\begin{proposition}\label{prop:consequence of Lebesgue points}
Let $u\in\BV(X)$. Then for $\mathcal H$-almost every $x\in X$ there exists a set $A_x\ni x$ with
density $1$ at $x$
such that if $y_k\to x$ with $y_k\in A_x$, then
\[
\min_{l_2\in \{1,\ldots,n\}}|u^{l_1}(y_k)-u^{l_2}(x)|\to 0
\]
for each $l_1=1,\ldots,n$.
\end{proposition}

From Theorem \ref{thm:fine continuity} we get the following strengthening of this result.
\begin{proposition}\label{prop:consequence of fine continuity}
Let $u\in\BV(X)$. Then for $\mathcal H$-almost every $x\in X$ there exists a set $A_x\ni x$ with
\[
\lim_{r\to 0^+}r\frac{\mathcal H_1(B(x,r)\setminus A_x)}{\mu(B(x,r))}=0
\]
such that if $y_k\to x$ with $y_k\in A_x$, then
\[
\min_{l_2\in \{1,\ldots,n\}}|u^{l_1}(y_k)-u^{l_2}(x)|\to 0
\]
for each $l_1=1,\ldots,n$.
\end{proposition}

\begin{proof}
In the proof of Theorem \ref{thm:fine continuity} we showed that for $\mathcal H$-almost every $x\in X$, there exists a $1$-finely open set $U\ni x$ such that if $y_k\to x$ with $y_k\in U$, then
\[
\min_{l_2\in \{1,\ldots,n\}}|u^{l_1}(y_k)-u^{l_2}(x)|\to 0
\]
for all $l_1=1,\ldots,n$. By using first Remark \ref{rmk:capacities and Hausdorff contents} and then \cite[Proposition 6.16]{BB}, we get for small $r>0$
\[
\mathcal H_1(B(x,r)\setminus U)\le C\capa_1(B(x,r)\setminus U)\le C\rcapa_1(B(x,r)\setminus U,B(x,2r)).
\]
Thus we can take $A_x=U$.
\end{proof}

Roughly speaking, if Proposition \ref{prop:consequence of Lebesgue points} says that the complement of $A_x$ cannot have significant ``volume'' close to $x$, Proposition \ref{prop:consequence of fine continuity} says that it cannot have significant ``surface area'' either.

The reason for considering more than two jump values is explained in the following example, which is essentially from \cite[Example 5.1]{LaSh}.

\begin{example}\label{ex:one dimensional space}
Consider the one-dimensional space
\[
X:=\{x=(x_1,x_2) \in \R^2:\,x_1=0\text{ or }x_2=0\}
\]
consisting of the two coordinate axes.
Equip this space with the Euclidean metric inherited from $\R^2$, and the 1-dimensional Hausdorff measure. 
This measure is doubling and supports a $(1,1)$-Poincar\'e inequality.
Moreover, we can take $\gamma=1/4$ in \eqref{eq:definition of gamma}, and then the number of jump values defined in \eqref{eq:definition of n jump values} is $n=1/\gamma=4$.
Let
\[
u:=\chi_{\{x_1>0\}}+2\chi_{\{x_2>0\}}+3\chi_{\{x_1<0\}}+4\chi_{\{x_2<0\}}.
\]
For brevity, denote the origin $(0,0)$ by $0$. Now $S_u=\{0\}$ with $\mathcal H(\{0\})=2$, and
$(u^1(0),u^2(0),u^3(0),u^4(0))=(1,2,3,4)$. The function
\[
x\mapsto \{u^1(x),u^2(x),u^3(x),u^4(x)\}
\]
is easily seen to be upper hemicontinuous everywhere (even with respect to the metric topology), but this would not be the case if we considered fewer that $4$ jump values.
\end{example}

The following very simple example demonstrates that we cannot in general have $1$-fine upper hemicontinuity at \emph{every} point.

\begin{example}
Let $X=[-1,1]\times [-1,1]\subset \R^2$ with the Euclidean distance and the $2$-dimensional Lebesgue measure $\mathcal L^2$. Since we now have $\gamma=1/2$ in \eqref{eq:definition of gamma}, see e.g. \cite[Theorem 3.59]{AFP}, we only consider the two jump values $u^{\wedge}$ and $u^{\vee}$ of a given $\BV$ function $u$, recall \eqref{eq:definition of n jump values}.
Denoting $x=(x_1,x_2)$, let 
\[
u(x_1,x_2)=
\begin{cases}
1,\qquad x_1<0,\\
0,\qquad x_1\ge 0,\,x_2\le 0,\\
2,\qquad x_1\ge 0,\,x_2>0.
\end{cases}
\]
Clearly $u\in\BV(X)$. For brevity, denote the origin $(0,0)$ by $0$. Take $\eps=1$. If $x\mapsto (u^{\wedge}(x),u^{\vee}(x))$ were $1$-finely continuous at the origin, there would exist a $1$-finely open set $U\ni 0$ with
\[
\min \{ |u^{\wedge}(y)-u^{\wedge}(0)|,\, |u^{\wedge}(y)-u^{\vee}(0)| \} <1
\]
for all $y\in U$, so necessarily $u^{\wedge}(y)=0$ or $u^{\wedge}(y)=2$ at these points.
However, $U$ must necessarily intersect $\{y=(y_1,y_2):\,y_1<0\}$, since
\[
\liminf_{r\to 0^+}r\frac{\rcapa_1(\{y_1<0\}\cap B(0,r),B(0,2r))}{\mathcal L^2(B(0,r))}>0.
\]
On the other hand, $u^{\wedge}(y)=1$ for all $y\in \{y_1<0\}$. 
Thus $1$-fine upper hemicontinuity fails at the origin. However, it does hold at every other point, and $\mathcal H(0)=0$.
\end{example}

Note that if $E\subset X$ and $u=\chi_E$, then $x\in I_E$ means that $u^\wedge(x)=u^\vee(x)=1$, $x\in O_E$ means that 
$u^\wedge(x)=u^\vee(x)=0$, and $x\in \partial^*E$ means that 
$u^\wedge(x)=0$ and $u^\vee(x)=1$.

The following example concerning the enlarged rationals illustrates the need to consider upper hemicontinuity with respect to the $1$-fine topology instead of the metric topology.

\begin{example}
Consider the Euclidean space $\R^2$.
Let $\{q_i\}_{i\in\N}$ be an enumeration of $\mathbb{Q}\times\mathbb{Q}\subset\R^2$,
and define
\[
E:=\bigcup_{i\in\N} B(q_i,2^{-i}).
\]
Clearly $\mathcal L^2(E)\le \pi$. By the lower semicontinuity and subadditivity of perimeter, see \eqref{eq:Caccioppoli sets form an algebra}, we can estimate
\[
P(E,\R^2)\le \sum_{i=1}^{\infty} P(B(q_i,2^{-i}),\R^2)\le 2\pi \sum_{i=1}^{\infty}2^{-i},
\]
so that $P(E,\R^2)<\infty$, and then also $\mathcal H(\partial^*E)<\infty$.
However, $\partial E=\R^2\setminus E$.
Thus, denoting $u:=\chi_E$, for every $x\in O_E$ there exists a sequence $y_k\to x$ with $y_k\in E\subset I_E$ such that 
\[
u^{\wedge}(y_k)=u^{\vee}(y_k)=1\not \rightarrow 0=u^{\wedge}(x)=u^{\vee}(x).
\]
Thus at almost every point $x\in \R^2\setminus E$, the function $y\mapsto \{u^{\wedge}(y),u^{\vee}(y)\}$ fails to be upper hemicontinuous with respect to the metric topology.
\end{example}

%

\end{document}